\documentclass[11pt]{amsart}

\usepackage[draft]{changes}

\definechangesauthor[color=red]{pe}
\definechangesauthor[color=red]{xm}
\definechangesauthor[color=red]{hx}
\newcommand{\stkout}[1]{\ifmmode\text{\sout{\ensuremath{#1}}}\else\sout{#1}\fi}
\setdeletedmarkup{\stkout{#1}}
\colorlet{Changes@Color}{red}

\usepackage{amssymb}

\usepackage{graphics}
\usepackage{graphicx}

\usepackage{amsmath}
\usepackage{amsthm}
\usepackage{amsfonts}
\usepackage{mathtools}
\usepackage{xcolor}
\usepackage[english]{babel}
\usepackage[margin=1.0in]{geometry}
\usepackage[colorlinks=true,
        linkcolor=blue]{hyperref}
\usepackage{bbm}
\usepackage{verbatim}
\usepackage{extarrows}
\usepackage{blkarray}

\usepackage[T1]{fontenc}

\numberwithin{equation}{section}

\newtheorem{prop}{Proposition}
\newtheorem{lemma}[prop]{Lemma}

\newtheorem{thm}[prop]{Theorem}
\newtheorem*{thm*}{Theorem}
\newtheorem{cor}[prop]{Corollary}
\newtheorem{conj}[prop]{Conjecture}

\numberwithin{prop}{section}

\newtheorem{defn}[prop]{Definition}
\theoremstyle{definition}

\newtheorem{rmk}[prop]{Remark}

\newtheorem{manualtheoreminner}{Theorem}

\newenvironment{manualtheorem}[1]{%
	\renewcommand\themanualtheoreminner{#1}%
	\manualtheoreminner
}{\endmanualtheoreminner}

\definecolor{c1}{rgb}{0.2,0.4,0.5}
\definecolor{c2}{rgb}{0.1,0.3,0.5}
\definecolor{c3}{rgb}{0.2,0.7,0.5}
\usepackage{tikz}

\newcommand{\oo}[1]{\overline{#1}}

\newcommand{\nab}{\nabla}

\newcommand{\bC}{\mathbb{C}}

\newcommand{\bB}{\mathbb{B}}

\DeclareMathOperator{\Real}{Re}
\DeclareMathOperator{\Imaginary}{Im}

\DeclareMathOperator{\length}{length}

\begin{document}

\title[]{Algebraic Bergman kernels and finite type domains in $\bC^2$}

\begin{abstract}
Let $G \subset \mathbb{C}^2$ be a smoothly bounded  pseudoconvex domain and assume that the Bergman kernel of $ G $ is algebraic of degree $d$. We show that the boundary $\partial G $ is of finite type and the type $r$ satisfies $r\leq 2d$. The inequality is optimal as equality holds for the egg domains $\{|z|^2+|w|^{2s}<1\},$ $s \in \mathbb{Z}_+$, by D'Angelo's explicit formula for their Bergman kernels.
Our results imply, in particular, that a smoothly bounded  pseudoconvex domain $G \subset \mathbb{C}^2$ cannot have rational Bergman kernel unless it is strongly pseudoconvex and biholomorphic to the unit ball by a rational map. Furthermore,  we show that if the Bergman kernel of $G$ is rational of the form$\frac{p}{q}$, reduced to lowest degrees, 
then its rational degree $\max\{\deg p, \deg q \}\geq 6$. Equality is achieved if and only if $ G $ is biholomorphic to the unit ball by a complex affine transformation of $\mathbb{C}^2$.
\end{abstract}



\author [Ebenfelt]{Peter Ebenfelt}
\address{Department of Mathematics, University of California at San Diego, La Jolla, CA 92093, USA} \email{{pebenfelt@ucsd.edu}}

\author[Xiao]{Ming Xiao}
\address{Department of Mathematics, University of California at San Diego, La Jolla, CA 92093, USA}
\email{{m3xiao@ucsd.edu}}

\author [Xu]{Hang Xu}
\address{Department of Mathematics, University of California at San Diego, La Jolla, CA 92093, USA}
\email{{h9xu@ucsd.edu}}

\thanks{The first author was supported in part by the NSF grant DMS-1900955. The second author was supported in part by the NSF grants DMS-1800549 and DMS-2045104.}

\maketitle

\section{Introduction}
The Bergman kernel of a domain in complex space is an object of fundamental importance in several complex variables and complex geometry. 
A classical and important problem is to study the geometry of the domain and its boundary in terms of properties of the Bergman kernel and metric. For example, as a consequence of a series of papers \cite{FuWo, NeSh, HX}, the unit ball may be characterized among bounded strongly pseudoconvex domains $\Omega$ in $\bC^n$ in terms of the K\"ahler-Einstein property of its Bergman metric. This is the well-known Cheng Conjecture. Another example is the Ramadanov Conjecture, which asserts that the logarithmic term in Fefferman's asymptotic expansion of the Bergman kernel vanishes if and only if the boundary $\partial\Omega$ is spherical. The Ramadanov Conjecture has been proved for domains in $\bC^2$ (\cite{Graham1987b}, see also \cite{CE2, CE1} for a strong version) but is still open for $n\geq 3$.
In \cite{EbenfeltXiaoXu2020algebraicity}, the current authors introduced a new characterization of the unit ball $\bB^2$ in $\bC^2$ in terms of algebraicity and rationality of the Bergman kernel.

\begin{manualtheorem} {1} [\cite{EbenfeltXiaoXu2020algebraicity}] \label{thm old}
{\em Let $G$ be a bounded domain in $\mathbb{C}^2$ with smooth strongly pseudoconvex boundary. Then the Bergman kernel of $G$ is rational (resp.\ algebraic) if and only if there is a rational (resp.\ algebraic) biholomorphic map from $G$ to $\mathbb{B}^2$.}
\end{manualtheorem}

It is in general a difficult problem to compute explicit formulas for the Bergman kernel and such formulas have only been successfully obtained in very special situations. The most well-known case of a closed expression for the Bergman kernel is that of the unit ball $\bB^n$:
\begin{equation*}
	K_{\mathbb{B}^n}(z, \bar{z})=\frac{n!}{\pi^n} \frac{1}{(1-|z|^2)^{n+1}}.
\end{equation*}
We note that $K_{\mathbb{B}^n}(z, z)$ is rational and has rational degree $2n+2$ (see Remark \ref{rmk rational}).
In 1978, D'Angelo \cite{DA78}  considered the egg domain
\begin{equation}\label{Egg}
E_s =\{|z|^2+|w|^{2s}<1\}\subset \mathbb{C}^2,\quad s\in \mathbb{R}_+,
\end{equation}
and obtained an explicit formula for its Bergman kernel:
\begin{equation}\label{BK for egg domain}
K_{E_s}\bigl( (z,w), \oo{(z,w)} \bigr)=\sum_{k=0}^2 c_k \frac{(1-|z|^2)^{-2+\frac{k}{s}}}{\bigl((1-|z|^2)^{\frac{1}{s}}-|w|^2 \bigr)^{1+k}},
\end{equation}
where $c_0=0, c_1=\frac{1}{\pi^2}\cdot\frac{s-1}{s}$, and $c_2=\frac{1}{\pi^2}\cdot \frac{2}{s}$. We note that when $s \in \mathbb{Z}_+$, the egg domain $E_s$ is smoothly bounded and its boundary is of finite type with type $r=2s$. Furthermore, in this case the Bergman kernel $K_{E_s}$ is algebraic of degree $s$ (see Definition \ref{def algebraic function} and \ref{def algebraic degree and total degree}, and Remark \ref{rmk16}). Since the work of D'Angelo \cite{DA78}, many authors have contributed explicit calculations of Bergman kernels. We refer the reader to \cite{DA94, BoFuSt, Pa08, Huo, CKMM} and the references therein.

The first main result of this paper concerns the characterizations in Theorem 1 above in the context of weakly pseudoconvex domains. It is clear from D'Angelo's formula for $K_{E_s}$ that Theorem 1 for algebraic Bergman kernels fails in this context. Instead, our first main result shows that the domain must be of finite type and a sharp bound on the type is given in terms of the algebraic degree of its Bergman kernel. In the case where the Bergman kernel is rational, the established bound implies that the domain is strongly pseudoconvex and the characterization of the unit ball in Theorem 1 remains.

To describe our results more precisely, we first recall some background materials on algebraic functions. For more details, the reader is referred to \cite[Chapter 5.4]{BER}.


\begin{defn}[Algebraic functions]\label{def algebraic function}
{\rm	A real-analytic function $f(t)$ in $U\subseteq \mathbb{R}^n$ is} algebraic, {\rm if there exists a non-trivial polynomial $P(t,Y)=\sum_{j=0}^N\alpha_j(t) Y^j \in \mathbb{R}[t,Y]$ such that $P(t,f(t))\equiv 0$ for all $t=(t_1,\cdots,t_n)\in U$.}
\end{defn}

\begin{rmk}
	Given a non-trivial algebraic function $f(t)$ on $U$, the ideal $I=\{P(t,Y)\in \mathbb{R}[t,Y]: P(t,f(t))\equiv 0 \mbox{ on } U \}$ is a principal ideal in $\mathbb{R}[t,Y]$. Thus we can write $I=(P^*(t,Y))$ for some generator $P^*\in \mathbb{R}[t,Y]$. We shall call $P^*$ a \emph{minimal polynomial} of $f$. The minimal polynomial is unique up to scaling by nonzero real numbers.
\end{rmk}

\begin{defn}[Algebraic/Total degree]\label{def algebraic degree and total degree}
{\rm	Let $f$ be a non-trivial algebraic function in $U\subset \mathbb{R}^n$. Let $P^*=P^*(t,Y)\in \mathbb{R}[t,Y]$ be the minimal polynomial of $f(t)$.	
	\begin{itemize}
		\item[(i)] The {\em algebraic degree} of $f$ is defined to be $d$, the degree of $P^*$ in $Y$. Equivalently, $d$ is the degree of the field extension $[\mathbb{R}(t, f(t)): \mathbb{R}(t)]$.
		\item[(ii)] The {\em total degree} of $f$ is defined to be the degree of $P^*$ as a polynomial in $(t,Y)$.
	\end{itemize}
}
\end{defn}

\begin{rmk}\label{rmk rational}
We note that an algebraic function $f$ is rational if and only if its algebraic degree is $1$.
In this case, we can write $f(t)=\frac{p(t)}{q(t)}$ for some $p(t), q(t)\in \mathbb{R}[t]$  reduced to lowest degree, i.e., $\gcd (p(t),q(t))=1$ in $\mathbb{R}[t]$. Then $q(t)Y-p(t)$ is a minimal polynomial of $f$ and the total degree of $f$ is $\max\{1+\deg q, \deg p \}$. The integer $\max\{\deg q, \deg p \}$ is the {\em rational degree} of $f$.
\end{rmk}

We also recall that the two standard notions of {\em finite type} (Kohn/Bloom--Graham and D'Angelo) of a smooth real hypersurface $M$ at a point $\xi$ coincide in $\bC^2$. The reader is referred to \cite{BER} for this fact and precise definitions. Roughly, the type of $M$ at $\xi$ is the maximal order of contact with a complex curve (D'Angelo) or the minimal number of commutators of complex tangential vector fields needed to span the full tangent space (Kohn/Bloom--Graham).

We shall now present our main results. The first result concerns the algebraic degree of the Bergman kernel.
\begin{thm}\label{main thm}
	Let $ G \subset \mathbb{C}^2$ be a smoothly bounded pseudoconvex domain and assume the Bergman kernel $K$ of $ G $ is algebraic. Then the boundary $\partial G $ is real algebraic and of finite type. Moreover, the algebraic degree $d$ of $K$ and the type $r(\xi)$ of $\partial  G $ at $\xi\in\partial G $ satisfy
		\begin{equation}\label{main inequality}
		\max_{\xi\in \partial G }\, r(\xi)\leq 2d.
		\end{equation}
\end{thm}

\begin{rmk}\label{rmk16}
The inequality \eqref{main inequality} is optimal. This can be seen from the unit ball $\mathbb{B}^2$, which is strongly pseudoconvex (i.e., $r(\xi)=2$ for any $\xi\in \partial \mathbb{B}^n$) and has a rational Bergman kernel (i.e., $d=1$). More generally, consider the egg domain $G=E_s$ given by \eqref{Egg} for an integer $s\geq 2$. Then $\partial G $ is strongly pseudoconvex at points with $w\neq 0$ and of finite type $2s$ at points with $w=0$. Thus, $\max_{\xi\in \partial G }\, r(\xi)=2s$. Moreover, D'Angelo's explicit formula (\ref{BK for egg domain}) gives the Bergman kernel $K$ of $G$.
To compute the algebraic degree $d$ of $K$, we let $F=\mathbb{R}(\Real{z},\Imaginary{z},\Real{w},\Imaginary{w})$ be the field of rational functions in $\Real{z},\Imaginary{z}$ and $\Real{w},\Imaginary{w}$. Then by \eqref{BK for egg domain} the Bergman kernel $K$ belongs to the extension field $F((1-|z|^2)^{\frac{1}{s}})$ and it is easy to see that the degree of the field extension $[F((1-|z|^2)^{\frac{1}{s}}):F]$ is $s$. Consequently, $d= s=\frac{1}{2}\max_{\xi\in \partial G }\, r(\xi)$ and, hence, equality holds in \eqref{main inequality} for any positive integer $s$ and $G=E_s$.
\end{rmk}

\begin{rmk}
One cannot expect to give an upper bound for the algebraic degree of the Bergman kernel in terms of $\max_{\xi\in \partial G } r(\xi)$ since the latter is a biholomorphic invariant of the domain $G$. Indeed, apply an algebraic biholomorphism $F$ in a neighborhood of the closed ball $\oo{\mathbb{B}^2}$ and consider the image $G=F(\mathbb{B}^2)$ and its Bergman kernel $K_G$. By the transformation law of the Bergman kernel, we can make the algebraic degree of $K_G$ arbitrarily large by choosing an appropriate map $F$.
\end{rmk}

If the Bergman kernel $K$ in Theorem \ref{main thm} is rational, then the conclusion of the theorem implies that $r(\xi)=2$ for all $\xi\in \partial G $. Hence, $\partial G $ is strongly pseudoconvex. By Theorem 1 above, there is a rational biholomorphism from $ G $ to the unit ball $\mathbb{B}^2$. Thus, an immediate corollary of Theorem \ref{main thm} is the following extension of Theorem 1 to the weakly pseudoconvex case.

\begin{cor} \label{main_cor}
Let $G$ be a bounded pseudoconvex domain in $\mathbb{C}^2$ with smooth boundary. Then the Bergman kernel of $G$ is rational if and only if $G$ is strongly pseudoconvex and there is a rational biholomorphism $G\to\mathbb{B}^2$.
\end{cor}

Our second main result discusses the total degree of an algebraic Bergman kernel. Throughout the paper, by a complex affine transformation of $\mathbb{C}^n,$ we mean a map $\Phi(z)= Mz+\xi$ from $\mathbb{C}^n$ to $\mathbb{C}^n$ with $M \in\text{GL}(n, \mathbb{C})$ and $\xi\in\mathbb{C}^n.$

\begin{thm}\label{total degree thm in C2}
	Let $G\subset \mathbb{C}^2$ be a smoothly bounded pseudoconvex domain and $K$ its Bergman kernel. If $K$ is algebraic, then:
	\begin{itemize}
		\item [(a)] The total degree of $K$ is at least $7$.
		\item [(b)] The total degree of $K$ equals $7$ if and only if $G$ is the unit ball up to a complex affine transformation of $\mathbb{C}^2$. In this case, $K$ is a rational function with rational degree $6$.
	\end{itemize}
\end{thm}

\begin{rmk}\label{total degre gap rmk}
	In fact, we can prove that if the total degree of $K$ is less than $10$, then $G$ is the unit ball up to a complex affine transformation of $\mathbb{C}^2$.
Consequently, there is no smoothly bounded pseudoconvex domain in $\mathbb{C}^2$ admitting an algebraic Bergman kernel of total degree $d$ with $7< d <10$. For a proof of this remark, see the end of Section 3.
\end{rmk}

Suppose the Bergman kernel $K$ is rational of the form $K=\frac{p}{q}$, with $p$ and $q$ reduced to lowest degree. In this case, the total degree of $K$ is $\max\{1+\deg q, \deg p\}$. Part (a) of Theorem \ref{total degree thm in C2} implies that the rational degree of $K$, $\max\{\deg q, \deg p\}$, is at least $6$. Moreover, if the rational degree of $K$ equals $6$, then the total degree of $K$ is $7$ and part (b) of Theorem \ref{total degree thm in C2} ensures that $G$ is the unit ball up to a complex affine transformation. Combining Theorem \ref{total degree thm in C2} with Corollary \ref{main_cor}, we therefore obtain the following extension of Corollary \ref{main_cor}:

\begin{cor}\label{total degree corollary in C2}
	Let $G\subset \mathbb{C}^2$ be a smoothly bounded pseudoconvex domain and $K$ its Bergman kernel. If $K$ is rational, then there is a rational biholomorphism $\Phi\colon G\to \mathbb{B}^2$ and:
	\begin{itemize}
		\item [(a)] The rational degree of $K$ is at least $6$.
		\item [(b)] The rational degree of $K$ equals $6$ if and only if $\Phi$ can be taken as a complex affine transformation of $\mathbb{C}^2$.
	\end{itemize}
\end{cor}

\begin{rmk}
It follows from Remark \ref{total degre gap rmk} that if the rational degree of $K$ is less than $9$, then $G$ is the unit ball up to a complex affine transformation of $\mathbb{C}^2$. Indeed, if $\max\{\deg p, \deg q \}<9$, then the total degree of $K$ is less than $10$.  Therefore the conclusion follows from Remark \ref{total degre gap rmk}.   Consequently, there is no smoothly bounded pseudoconvex domain in $\mathbb{C}^2$ whose Bergman kernel is rational of rational degree $d$ with $6 <d < 9$.
\end{rmk}

This paper is organized as follows. In Section \ref{Sec proof of main thm}, we give a proof of Theorem \ref{main thm}. In Section \ref{Sec proof of total degree thm}, we show Theorem \ref{total degree thm in C2}. In Section \ref{Sec generalization}, we consider the higher dimensional case and establish some partial results in this case.

\section{Proof of Theorem \ref{main thm}}\label{Sec proof of main thm}
We prove the following lemma as a preparation.
\begin{lemma}\label{vanishing lemma lem}
	Let $r\geq 1$. For $0\leq j\leq r-1$, suppose $A_j(t): (-\varepsilon,\varepsilon) \rightarrow \mathbb{R}$ are smooth functions for some $\varepsilon>0$; $c_j$ are nonzero real numbers; $B_j(t): (0,\varepsilon) \rightarrow \mathbb{R}$ are functions such that $B_j(t) \rightarrow 0$ as $t\rightarrow 0^+$. Assume
	\begin{equation}\label{vanishing lemma eq}
		\sum_{j=0}^{r-1} A_j(t)t^{\frac{j}{r}} \bigl( c_j+B_j(t)\bigr) \equiv 0 \quad \mbox{on }~ (0,\varepsilon).
	\end{equation}
	Then each $A_j(t)$ for $0\leq j\leq r-1$ vanishes to infinite order at $0$.
\end{lemma}

\begin{proof}
	The conclusion is trivial when $r=1$. We therefore assume $r\geq 2.$
	Assume not all $A_j$ vanish to infinite order. Then we can find the smallest $m\in \mathbb{Z}_{\geq 0}$ such that the $m$th order derivative of $A_j$ at $0$ is nonzero for at least one $0\leq j\leq r-1$ and let $j_0$ be the smallest such $j$. Then we have
	\begin{align*}
		A_{j_0}(t)=&a_{j_0}t^m+O(t^{m+1}) \quad \mbox{ for some } a_{j_0}\neq 0;
		\\
		A_j(t)=&O(t^{m+1}) \quad \mbox{ for } 0\leq j<j_0;
		\\
		A_j(t)=&O(t^m) \quad \mbox{ for } j_0 < j \leq r-1.
	\end{align*}
	By plugging these equations into \eqref{vanishing lemma eq}, we obtain
	\begin{align*}
		a_{j_0}t^{m+\frac{j_0}{r}}\bigl( c_{j_0}+B_{j_0}(t)\bigr) +O(t^{m+1})+O(t^{m+\frac{j_0+1}{r}})=0, \quad \mbox{as } t\rightarrow 0^+.
	\end{align*}
	As $B_{j_0}(t)=o(1)$ when $t\rightarrow 0^+$, it follows that
	\begin{equation*}
		a_{j_0}c_{j_0}t^{m+\frac{j_0}{r}}=o(t^{m+\frac{j_0}{r}}), \quad \mbox{as } t\rightarrow 0^+.
	\end{equation*}
	This is contradicting to the fact $a_{j_0}c_{j_0}\neq 0$ and thus the proof is completed.
\end{proof}

We are now ready to prove Theorem \ref{main thm}.
\begin{proof}[Proof of Theorem \ref{main thm}]
	Since $G$ is a smoothly bounded pseudoconvex domain, the Bergman kernel $K(z,\bar{z})\rightarrow \infty$ as $z\rightarrow \partial G$ by \cite{Oh93}. And since the Bergman kernel $K(z,\bar{z})$ is algebraic, the proof of Proposition 5.1 in \cite{EbenfeltXiaoXu2020algebraicity} yields that $\partial G$ is real algebraic and, hence,  real analytic. Now that $\partial G$ is a compact real analytic hypersurface, it is of finite type.
	
	By assumption, the Bergman kernel $K$ is of algebraic degree $d$. We write the minimal polynomial of $K$ as
	\begin{equation*}
		P(z, \bar{z}, Y)=\sum_{j=0}^d\alpha_j(z,\bar{z})\, Y^j,
	\end{equation*}
	where $\alpha_j$ for $0\leq j\leq d$ are real polynomials and $\alpha_d\not\equiv 0$. Then
	\begin{equation}\label{BK polynomial equation}
		\alpha_d(z,\bar{z})K^d+\cdots+\alpha_0(z,\bar{z})\equiv 0.
	\end{equation}	
Moreover, it is proved in the recent paper of Hsiao and Savale \cite[proof of Theorem 2]{HsSa20} that for any fixed $\xi \in \partial G $ of finite type $r=r(\xi)$, the Bergman kernel $K(z,\bar{z})$ has the following expansion near $\xi$ in $G$ along any real line segment $L$ that intersects $\partial G$  transversally at $\xi$:
	\begin{equation}\label{BK asymptotic}
		K(z,\bar{z})=\rho^{-2-\frac{2}{r}} \Bigl( \sum_{j=0}^{r-1} a_j\rho^{\frac{j}{r}}+C(z,\bar{z})\Bigr)+B(z,\bar{z}) \log\rho,
	\end{equation}
in terms of a defining function $\rho$ for $ G $ (with $\rho>0$ in $G$). Here, the $a_j$ are constants with $a_0>0$, $B(z,\bar{z})$ is a real smooth function in a neighborhood of $\xi$, and $C(z,\bar{z})$ is a function in $G$ near $\xi$ satisfying $C(z,\bar{z})=O(\rho)$ as $z \rightarrow \xi$ along $L$. The $a_j$, $B$, and $C$ may all depend on the line $L$.
We choose a real line segment $L$ such that $\alpha_d(z,\bar{z})|_L$ has finite vanishing order at $\xi$ and parametrize $L$ by a coordinate $t\in \mathbb{R}$ such that $t(\xi)= 0$ and $L|_{t\in (0,\varepsilon)}\subset G $ for some $\varepsilon>0$. Set
	\begin{equation*}
		\phi(t)=\rho|_L, \qquad b(t)=B|_L, \qquad c(t)=C|_L.
	\end{equation*}
	Since $\rho$ is a defining function and $L$ intersects $\partial G $ transversally, we can write $\phi(t)=t\varphi(t)$ for some smooth function $\varphi$ at $0$ with $\varphi(0)\neq 0$. We also have $c(t)=O(t)$. If we write
	\begin{equation*}
		\psi=\varphi^{-2-\frac{2}{r}}, \quad \psi_j=a_j\varphi^{\frac{j}{r}} \quad \mbox{for } 0\leq j\leq r-1,
	\end{equation*}
	then $\psi$, $\psi_j$ are all smooth at $0$, and $\psi, \psi_0$ are nonvanishing at $0$.
	By restricting \eqref{BK asymptotic} to the line segment $L$, we get
	\begin{align*}
		K|_L=t^{-2-\frac{2}{r}}\psi\Bigl( \sum_{j=0}^{r-1}\psi_j t^{\frac{j}{r}}+c(t)\Bigr)+b(t)\bigl(\log\varphi+\log t \bigr)
		=t^{-2-\frac{2}{r}}\bigl( \psi\psi_0+O(t^{\frac{1}{r}})\bigr).
	\end{align*}
	By taking the Taylor expansion of $\psi\psi_0$ at $t=0$, we further have
	\begin{equation}\label{BK asymptotic on L}
		K|_L=t^{-2-\frac{2}{r}}\bigl(c_0+O(t^{\frac{1}{r}})\bigr),
	\end{equation}
	where $c_0=\psi(0)\psi_0(0)\neq 0$.
	We now restrict \eqref{BK polynomial equation} to the line segment $L$. Denoting $\beta_j(t)=\alpha_j|_L$ for $0\leq j\leq d$, we have
	\begin{equation*}
		\beta_d(t) (K|_L)^d+\cdots+\beta_0(t)=0.
	\end{equation*}
	Substituting \eqref{BK asymptotic on L} into this equation yields
	\begin{equation*}
		\beta_d(t) t^{-(2+\frac{2}{r})d} \bigl( c_0+O(t^{\frac{1}{r}})\bigr)^d+\cdots+\beta_0(t)=0 \quad \mbox{ on } (0,\varepsilon).
	\end{equation*}
	Multiplying through by $t^{(2+\frac{2}{r})d}$, we obtain
	\begin{equation*}
		\beta_d(t) \bigl( c_0^d+O(t^{\frac{1}{r}})\bigr)+\beta_{d-1}(t) t^2t^{\frac{2}{r}} \bigl( c_0^{d-1}+O(t^{\frac{1}{r}})\bigr)+\cdots+\beta_0(t)t^{2d}t^{\frac{2d}{r}}=0\quad \mbox{ on } (0,\varepsilon).
	\end{equation*}
	Assume $2d<r(\xi)$. Then, by Lemma \ref{vanishing lemma lem}, all $\beta_j$ must vanish to infinite order at $t=0$. This contradicts the choice of the line $L$, since we required that  $\beta_d=\alpha_d|_L$ has finite vanishing order at $\xi$ (i.e., $t=0$). Thus we must have
	\begin{equation*}
		2d\geq r=r(\xi).
	\end{equation*}
	Since $\xi$ was an arbitrary point on $\partial G $, the conclusion of Theorem \ref{main thm} follows.
\end{proof}

\section{Proof of Theorem \ref{total degree thm in C2}}\label{Sec proof of total degree thm}

As a preparation  for the proof of Theorem \ref{total degree thm in C2},  we will first prove a general result for bounded pseudoconvex domains of any dimension.
For this, we introduce
\begin{equation}\label{ellipsoid parameter space}
\mathcal{A}=\{A=(A_1,\cdots, A_n): 0\leq A_1\leq \cdots \leq A_n<\tfrac{1}{2} \}.
\end{equation}
For each $A\in \mathcal{A}$, we can define the associated real ellipsoid as
\begin{equation}\label{real ellipsoid eq}
E(A)=\{f_A(z,\bar{z}):=1-|z|^2-\sum_{j=1}^n A_j(z_j^2+\oo{z_j}^2)>0 \}.
\end{equation}
\begin{thm}\label{total degree thm}
	Let $G\subset \mathbb{C}^n (n\geq 2)$ be a smoothly bounded  pseudoconvex domain. Let $K$ be the Bergman kernel of $G$. If $K$ is algebraic, then
	\begin{itemize}
		\item [(a)] The total degree of $K \geq 2n+3$.
		\item [(b)] If the total degree of $K=2n+3$, then $G$ is a real ellipsoid $E(A)$, for some $A\in \mathcal{A}$, up to a complex affine transformation of $\mathbb{C}^n$.
	\end{itemize}
\end{thm}

For the proof of Theorem \ref{total degree thm}, we shall need some preliminary results. Let $\rho$ be a smooth defining function of $G$ such that $G=\{\rho>0\}$. Since $G$ is smoothly bounded, there exists some strongly pseudoconvex point $p\in \partial G$. We may localize Fefferman's asymptotic expansion of the Bergman kernel of $G$ near this point (see Section 9 in \cite{Ka89} or Section 3 of \cite{HuLi}). Thus, in a sufficiently small neighborhood $U$ of $p$, we may express the Bergman kernel $K$ for $z\in U\cap G$ as follows
\begin{equation}\label{Fefferman expansion}
	K(z,\bar{z})=\frac{\phi(z,\bar{z})}{\rho^{n+1}(z)}+\psi(z,\bar{z})\log\rho(z),
\end{equation}
where $\phi$ and $\psi$ are smooth functions on $\oo{U}$.
In what follows, for $f\in C^{\infty}(\oo{U})$ and $k\in \mathbb{Z}^+$ we say $f=O(\rho^k)$ on $U$ if there exists some $C>0$ such that $|f(z)|\leq C|\rho(z)|^k$ for any $z\in U$. Equivalently, $f=O(\rho^k)$ means that all the partial derivatives of $f$ up to order $(k-1)$ vanish on $\partial G\cap U$. We say $f$ vanishes to infinite order at $\partial G\cap U$ if $f=O(\rho^k)$ for every $k$.

\begin{lemma}\label{lemma extension of 1/K}
	If $\psi$ vanishes to infinite order at $\partial G\cap U$, then $\frac{1}{K}$ extends smoothly across $\partial G \cap U$ and $\frac{1}{K}=O(\rho^{n+1})$ on $U$.
\end{lemma}
\begin{proof}
	By the Fefferman expansion \eqref{Fefferman expansion}, we have
	\begin{equation*}
		\frac{1}{K}=\frac{\rho^{n+1}}{\phi+\psi\rho^{n+1}\log\rho}.
	\end{equation*}
	Note that $\phi$ is nonvanishing on $\partial G \cap U$ (cf. Theorem 3.5.1 in \cite{Ho65}).  Since $\psi$ vanishes to infinite order at $\partial G\cap U$, it follows that $\phi+\psi\rho^{n+1}\log\rho$ extends smoothly across $\partial G \cap U$ such that $\phi+\psi\rho^{n+1}\log\rho$ is nonvanishing on $\partial G \cap U$. The conclusion now follows.
\end{proof}

\begin{rmk}\label{rem:lemma extension of 1/K} If $K$ is algebraic, then $\psi$ vanishes to the infinite order at $\partial G \cap U$ by \cite{EbenfeltXiaoXu2020algebraicity} (see Step 2 in the proof of Theorem 1.1. Although it deals with $2-$dimensional case there, the same argument indeed works for all dimensional cases). Consequently, by Lemma \ref{lemma extension of 1/K}, $\frac{1}{K}$ extends smoothly across $\partial G \cap U$ and $\frac{1}{K}=O(\rho^{n+1})$ on $U$.
\end{rmk}

The following proposition is surely well-known to experts in the field. Since we could find no references for it, however, we include a proof here for the reader's convenience.
\begin{prop}\label{bG is connected}
	If $G\subset \mathbb{C}^n$ $(n\geq 2)$ is a smoothly bounded pseudoconvex domain, then its boundary $\partial G$ is connected.
\end{prop}

\begin{proof}
	Let $M$ be one connected component of $\partial G$. Then $M$ is closed in $\partial G$. Since $\partial G$ is a compact smooth hypersurface in $\mathbb{C}^n$, so is $M$. By the Jordan-Brouwer separation theorem (see \cite{GuiPo}), $M$ separates $\mathbb{C}^n$ into two open connected components, the ``outside'' $D_{\text{out}}$ and the ``inside'' $D_{\text{in}}$. Moreover, $D_{\text{in}}$ is bounded with $M$ as its boundary. Since $G$ is connected and is disjoint from $M$, either $G \subset D_{\text{out}}$ or $G \subset D_{\text{in}}$. 
	
	We first consider the case when $G\subseteq D_{\text{out}}$. For each point $p\in M \subseteq \partial G$, there exists some $\varepsilon_p>0$ such that one of the one-sided neighborhoods, $B_{\varepsilon_p}(p)\cap D_{\text{in}}$ or $B_{\varepsilon_p}(p)\cap D_{\text{out}}$, is contained in $G$. By the assumption $G\subseteq D_{\text{out}}$, we get $B_{\varepsilon_p}(p)\cap D_{\text{out}}\subset G$. Since $M$ is compact, 
	there is some $\varepsilon>0$ such that the outside $\varepsilon$ neighborhood of $M$, $T_{\varepsilon}=\{z: d(z, M)<\varepsilon\} \cap D_{\text{out}}$, is contained in $G$. 
	Let $\Omega=G \cup \overline{D_{\text{in}}},$ which is an open set. By Hartogs's extension theorem, since $\Omega \setminus \overline{D_{\text{in}}}=G$ is connected, every holomorphic function on $G$ must extends to a holomorphic function on $\Omega.$ This contradicts to the pseudoconvexity of $G$ 
	
	Hence we must have $G\subseteq D_{\text{in}}$, by a similar argument, we can choose some $\varepsilon>0$ such that the inside $\varepsilon$ neighborhood of $M$, $P_{\varepsilon}=\{z: d(z, M)<\varepsilon\} \cap D_{\text{in}}$, is contained in $G$. Thus $K=D_{\text{in}} \setminus G$ is compact. By Hartogs's extension theorem again, since $D_{\text{in}}\setminus K=G$ is connected, we get every holomorphic function on $G$ extends to a holomorphic function on $D_{\text{in}}.$
	Hence we must have $G=D_{\text{in}}$ by the pseudoconvexity of $G$. In this case, $\partial G=M$. The proof is thus completed.
\end{proof}

We are now ready to prove Theorem \ref{total degree thm}.
\begin{proof}[Proof of Theorem \ref{total degree thm}]
As in the proof of Theorem \ref{main thm}, since $G$ is a smoothly bounded pseudoconvex domain, $K(z,\bar{z})\rightarrow \infty$ as $z\rightarrow \partial G$ by \cite{Oh93}. By the algebraicity of $K$ and the same proof as that of  Proposition 5.1 in \cite{EbenfeltXiaoXu2020algebraicity}, $\partial G$ is real algebraic, and  thus real analytic.
We pick a minimal polynomial $P(z,\bar{z},Y)\in\mathbb{C}[z,\bar{z},Y]$ of $K$,
	\begin{equation*}
		P=\sum_{j=0}^N a_j(z,\bar{z}) Y^j,
	\end{equation*}
where $N \geq 1,$ and $a_j'$s are real polynomials with $a_N\not\equiv 0$. Then we have
	\begin{equation*}
		a_N(z,\bar{z}) K^N(z,\bar{z})+\cdots+a_0(z,\bar{z})=0 \quad \mbox{ for any } z\in G.
	\end{equation*}
	Since $K>0$ in $G$, we can divide both sides by $K^N$ and solve for $a_N$:
	\begin{equation*}
		a_N=\frac{1}{K}\Bigl(-a_{N-1}-a_{N-2}\frac{1}{K}-\cdots-a_0\frac{1}{K^{N-1}}\Bigr).
	\end{equation*}
	As before, we denote by $\rho$ a smooth defining function of $G$. By Remark \ref{rem:lemma extension of 1/K}, we get $a_N=O(\rho^{n+1})$ on $U$, where $U$ is a neighborhood of some strongly pseudoconvex point $p\in \partial G$. Thus, all partial derivatives of $a_N$ up to order $n$ vanish on $\partial G\cap U$. Since $\partial G$ is real analytic and connected (see Proposition \ref{bG is connected}), these partial derivatives vanish on $\partial G$. Therefore, $a_N=O(\rho^{n+1})$ on a neighborhood of $\partial G$.
	
	Set
	\begin{equation*}
		\mathcal{I}=\{a\in \mathbb{C}[z,\bar{z}]: a\equiv 0 \mbox{ on } \partial G, \mbox{ and } \oo{a}=a \}\subset \mathbb{R}[\Real z, \Imaginary z].
	\end{equation*}
	
	Clearly, $\mathcal{I}$ is an ideal in the polynomial ring $\mathbb{R}[\Real z, \Imaginary z]$. We have  $\mathcal{I}\neq \{0\}$ since $a_N$ is a nonzero polynomial in $\mathcal{I}$. We take a nonzero polynomial $r\in \mathcal{I}$ such that $r(z,\bar{z})$ has the smallest degree among all nonzero polynomials in $\mathcal{I}$. It is clear that $\nab r (z,\bar{z})\not\equiv 0$ on $\partial G$. Otherwise, it contradicts to the choice of $r$. Consequently, $\nab r(z,\bar{z})\neq 0$ for a generic point $z \in \partial G$ since $\partial G$ is connected and algebraic.
	
	Let $r(z,\zeta)$ for $z,\zeta\in \mathbb{C}^n$ be the complexification of $r(z,\bar{z})$.
	
	\textbf{Claim 1.} $r(z,\zeta)$ is irreducible in $\mathbb{C}[z,\zeta]$. 	
		
	{\bf Proof of Claim 1.} Suppose not. Then there exists polynomials $a(z,\zeta), b(z,\zeta)\in \mathbb{C}[z,\zeta]$ such that
	\begin{equation*}
		r(z,\zeta)=a(z,\zeta) b(z,\zeta) \quad \mbox{ with } 1\leq \deg a, \deg b<\deg r.
	\end{equation*}
	Thus,
	\begin{equation*}
		a(z,\bar{z})b(z,\bar{z})\equiv 0 \quad \mbox{ for any } z\in \partial G.
	\end{equation*}
	By Proposition \ref{bG is connected}, $\partial G$ is connected, and thus either $a(z,\bar{z})\equiv 0$ or $b(z,\bar{z})\equiv 0$ on $\partial G$. Without loss of generality, we suppose $a(z,\bar{z})\equiv 0$ on $G$. Then either $\Real a$ or $\Imaginary a$ is a nonzero polynomial that vanishes on $\partial G$. But $\deg (\Real a), \deg(\Imaginary a)\leq \deg a<\deg r$. This contradicts the choice of $r$. So $r(z,\zeta)$ is irreducible in $\mathbb{C}[z,\zeta]$ as claimed. \qed
	
	\textbf{Claim 2.} $\mathcal{I}=(r)$, i.e., $\mathcal{I}$ is a principal ideal generated by $r(z,\bar{z})$.
	
  {\bf Proof of Claim 2.} Take $a(z,\bar{z})\in \mathcal{I}$ and let $a(z,\zeta)$ be its complexfication. In particular, $a(z,\bar{z})\equiv 0$ on $\partial G$. Let $U_0$ be a small neighborhood of some  $z_0 \in \partial G$ such that $\nab r \neq 0$ in $U_0$. Consequently, $r(z,\bar{z})$ is a defining function of $\partial G$ in $U_0$ and
	\begin{equation*}
		a(z,\bar{z})=0 \mbox{ whenever } r(z,\bar{z})=0, \quad  \mbox{ for any } z\in U_0.
	\end{equation*}
By complexification, shrinking $U_0$ if necessary, we have
	\begin{equation*}
	a(z,\zeta)=0 \mbox{ whenever } r(z,\zeta)=0, \quad  \mbox{ for any } z,\zeta \in U_0.
	\end{equation*}
Since $r(z,\zeta)$ is irreducible in $\mathbb{C}[z,\zeta]$ by Claim 1, the regular part of the complex algebraic variety  $V=\{(z, \zeta) \in \mathbb{C}^{2n}: r(z,\zeta)=0\}$ is connected. Thus we further have $a(z,\zeta)=0$ on $V$.
By Hilbert's Nullstellensatz,  $a(z,\zeta)=r(z,\zeta)b(z,\zeta)$ for some $b\in \mathbb{C}[z,\zeta]$. In particular, we get $a(z,\bar{z})=r(z,\bar{z})b(z,\bar{z})$. Since $a(z,\bar{z})$ and  $r(z,\bar{z})$ are real polynomials, so is $b(z,\bar{z}).$ The claim is proved. \qed
	
	Now we turn our attention back to the real polynomial $a_N(z,\bar{z})$. Since $a_N(z,\bar{z})=O(\rho^{n+1})$, we iterated applications of Claim 2 yields 
	\begin{equation*}
		a_N(z,\bar{z})=r^{n+1}(z,\bar{z}) q_{n+1}(z,\bar{z}) \quad \mbox{ for some } q_{n+1}(z,\bar{z}) \in \mathbb{C}[z,\bar{z}].
	\end{equation*}
	
It is clear that $r$ cannot be linear, as $r$ vanishes on the compact real hypersurface $\partial G$, and, hence, $\deg r\geq 2$. It follows that $\deg a_N \geq 2(n+1)$ and the total degree of $K\geq 2(n+1)+N \geq 2n+3$. This proves part (a) of Theorem \ref{total degree thm}.
	
To prove (b), suppose the total degree of $K =2n+3$. Then by the preceding argument, we must have
	\begin{equation*}
		N=1,~ \deg a_1=2(n+1),~ \deg a_0\leq 2n+3 ~ \mbox{ and } \deg r=2.
	\end{equation*}
	In particular, $K$ is a rational function in this case. We decompose $r$ into homogeneous terms
	\begin{equation*}
		r(z,\bar{z})=r_2(z, \bar{z})+r_1(z, \bar{z})+r_0,
	\end{equation*}
	where $r_2$ and $r_1$ are respectively the quadratic terms and linear terms in $r$, and $r_0$ is the constant term in $r$.
	
	\textbf{Claim 3.} $r_2$ is either positive definite or negative definite.
	
{\bf Proof of Claim 3.} Note that in terms of variables $\Real z$ and $\Imaginary z$, $r_2$ is a quadratic form  over $\mathbb{R}$. Thus it can be diagonalized by some linear transformation in $\text{O}(2n)$. Since the claim we are proving is invariant under real linear changes of coordinates, we can assume $r_2$ is diagonalized. By completing squares and applying a translation, we can further assume that $r$ takes the following form:  
	\begin{equation*}
		r=\sum x_i^2-\sum y_j^2+\sum u_k+c.
	\end{equation*}
	Here $(x_i,y_j,u_k)$ with possibly some missing variables $v_l$ are the real coordinates of $\mathbb{R}^{2n}$ and $c$ is a real constant. Since $\deg r\geq 2$, the variables $x_i$ and $y_j$ cannot be both empty. By taking $-r$ instead of $r$ if needed, we can assume $\{x_i\}$ is nonempty.
	
	Suppose $r_2$ is neither positive definite nor negative definite. Then the variables $(y_j,u_k,v_l)$ are not empty. We note that, since
	$\partial G$ contains an open set of smooth points of $\{r=0\}$, by real analyticity and closedness of $\partial G$,  it must contain one connected component of the regular part of $\{r=0\}$. To seek a contradiction, we shall show that each connected component of the regular part of $\{r=0\}$ is unbounded. If this is proved, then $\partial G$ would contain an unbounded connected component, contradicting the compactness of $\partial G$.
	
	Now let's prove that each connected component of the regular part of $\{ r=0 \}$ is unbounded. We split the proof into several cases depending on the number of variables $y_j, u_k$ and $v_l$. First, if $\{v_l\}$ is nonempty, as $r$ is independent from $v_l$, each component of the regular part of $\{r=0\}$ is clearly unbounded. If $\{ u_k\}$ is nonempty, then $\{r=0\}$ is smooth everywhere. After applying a translation, we can make $c=0$. Thus every point of $\{r=0\}$ is connected to the origin and therefore $\{r=0\}$ is connected. The unboundedness of $\{r=0\}$ follows from the fact that each $u_k$ can take any value in $(-\infty, 0]$ as $\{x_i\}$ is nonempty.
	
	It remains to consider the case $\{u_k\}$ and $\{v_l\}$ are both empty. In this case, $\{x_i\}$ and $\{y_j\}$ are both nonempty. Then $r=0$ becomes
	\begin{equation*}
		\sum x_i^2+c=\sum y_j^2.
	\end{equation*}
In this case, it is also clear that  each connected component of the regular part is unbounded.  The claim is thus verified. \qed

By Claim 3, replacing $r$ by $-r$ if necessary, we have $r_2$ is a strictly positive definite real quadratic form on $\mathbb{C}^n$, by a complex linear change of coordinates (see \cite[Lemma 4.1]{We}), we can make
	\begin{equation*}
		r_2=\sum_{j=1}^n |z_j|^2+\sum_{j=1}^n A_j(z_j^2+\oo{z_j}^2),
	\end{equation*}
	where $A_j\in \mathbb{R}_{\geq 0}$ for $1\leq j\leq n$. The positive definiteness of $r_2$ yields $ 0 \leq A_j<\frac{1}{2}$ for $1\leq j\leq n$.
	
	By further taking a linear translation in the form
	\begin{equation*}
		(z_1,\cdots, z_n)\rightarrow (z_1-a_1,\cdots, z_n-a_n), \quad \mbox{ for some } a_1,\cdots, a_n \in \mathbb{C},
	\end{equation*}
	we can eliminate the linear terms in $r$. Thus we can write
	\begin{equation*}
		r=\sum_{j=1}^n |z_j|^2+\sum_{j=1}^n A_j(z_j^2+\oo{z_j}^2)+\tilde{r}_0,
	\end{equation*}
	where $\tilde{r}_0$ is a negative real number. By scaling, we can assume $\tilde{r}_0=-1$. 
Consequently, $G$ is a real ellipsoid $E(A_1, \cdots, A_n)$ after some complex affine transformation of $\mathbb{C}^n$.
\end{proof}

\begin{rmk}\label{rmk3d4}
	We would like to point out that the condition in part (b) of Theorem \ref{total degree thm} can be relaxed to that of the total degree of $K$ being $<3(n+1)+1$. Indeed, if we have the total degree of $K<3(n+1)+1$, then it follows that $\deg a_N < 3(n+1)$ and $\deg r<3$. This yields $\deg r=2$ and the same argument as in the above proof will conclude the result.
\end{rmk}

Now we shall prove Theorem \ref{total degree thm in C2}.
\begin{proof}[Proof of Theorem \ref{total degree thm in C2}]
Part (a) of Theorem \ref{total degree thm in C2} follows from part (a) of Theorem \ref{total degree thm}. To prove part (b), we note that by part (b) of Theorem \ref{total degree thm} $G$ is a real ellipsoid $E(A_1,A_2)$ by a complex affine transformation with $0\leq A_1 \leq A_2 \leq \frac{1}{2}$. Consequently, $G$ is strongly pseudoconvex. Then by Theorem 1 above (which is \cite[Corollary 1.4]{EbenfeltXiaoXu2020algebraicity}), the algebraicity of $K$ further implies that $G$ is biholomorphic to the unit ball $\mathbb{B}^2$. Therefore, $E(A_1, A_2)$ is biholomorphic to $\mathbb{B}^2$. By a theorem of Webster \cite{We}, we get $A_1=A_2=0,$ i.e., $E(A_1, A_2)=\mathbb{B}^2$. Hence the desired result in part (b) follows.
\end{proof}

The conclusion in Remark \ref{total degre gap rmk} then follows from the proof of Theorem \ref{total degree thm in C2} and Remark \ref{rmk3d4}.

\section{Generalization to higher dimensions}\label{Sec generalization}

We would like to point out that part (b) in Theorem \ref{total degree thm} does not seem to be optimal if we compare it with the two dimensional  case as in Theorem \ref{total degree thm in C2}. In this regard,  we formulate the following conjecture.

\begin{conj}\label{conjecture}
	Let $G\subset \mathbb{C}^n (n\geq 2)$ be a smoothly bounded  pseudoconvex domain. Let $K$ be the Bergman kernel of $G$. If $K$ is algebraic, then the total degree of $K=2n+3$ if and only if $G$ is the unit ball up to a complex affine transformation in $\mathbb{C}^n$.
\end{conj}

As proved in \cite{EbenfeltXiaoXu2020algebraicity} (see Section 5 of \cite{EbenfeltXiaoXu2020algebraicity}), if the Bergman kernel $K$ of $G$ is algebraic, then it has no logarithmic singularity in its Fefferman expansion (meaning the logarithmic term vanishes to infinite order at $\partial G$). Hence Conjecture \ref{conjecture} will follow from Theorem \ref{total degree thm} if one can prove that the Bergman kernel of a real ellipsoid $E(A)$ has no logarithmic singularity in its Fefferman expansion if and only if $A=0$. The latter statement is a version of the Ramadanov conjecture \cite{Ramadanov1981} for real ellipsoids.

 Hirachi  \cite{Hirachi93} has confirmed the Ramadanov conjecture for real ellipsoids when $|A|$ is small. With the help of Hirachi's result, we can confirm Conjecture \ref{conjecture} under the condition that $G$ is sufficiently close to the unit ball in the Hausdorff distance.

\begin{thm}\label{local rigidity}
	There exists $\delta>0$ such that the following holds.
	Let $G\subset \mathbb{C}^n (n\geq 2)$ be a smoothly bounded  pseudoconvex domain. Suppose the Hausdorff distance $d_H(G, \mathbb{B}^n)<\delta$. Let $K$ be the Bergman kernel of $G$ and assume $K$ is algebraic. Then the total degree of $K=2n+3$ if and only if $G$ is the unit ball up to a complex affine transformation in $\mathbb{C}^n$.
\end{thm}

\begin{proof}[Proof of Theorem \ref{local rigidity}]
	We shall need the following lemma, where the notation $\mathbb{B}^n_{r}=\{ z\in \mathbb{C}^n: |z|<r \}$ is used.
	
	\begin{lemma}\label{Hausdorff distance lemma}
		Let $\Omega \subset \mathbb{C}^n$ be a convex open set. For any $\varepsilon\in (0,\frac{1}{2})$, if
$d_H(\Omega, \mathbb{B}^n)<\varepsilon$, then $\mathbb{B}^n_{1-\varepsilon}\subset \Omega \subset \mathbb{B}^n_{1+\varepsilon}.
$
	\end{lemma}
	
	\begin{proof}
		If $d_H(\Omega, \mathbb{B}^n)<\varepsilon$, then the definition of Hausdorff distance yields that $\Omega$ is contained in the $\varepsilon$ neighborhood of $\mathbb{B}^n$, i.e., $\Omega \subset \mathbb{B}^n_{1+\varepsilon}$. It remains to prove the other containment $\mathbb{B}^n_{1-\varepsilon}\subset \Omega$. Since $\mathbb{B}^n$ and $\Omega$ are both bounded, nonempty and convex, by Theorem 18 in \cite{Wi07} we have $d_H(\mathbb{B}^n,\Omega)=d_H(\partial\mathbb{B}^n, \partial\Omega)$. It follows that $\partial\Omega$ is contained in the $\varepsilon$ neighborhood of $\partial\mathbb{B}^n$, i.e., $\partial\Omega\subset \{z\in \mathbb{C}^n: 1-\varepsilon<|z|<1+\varepsilon \}$. It follows that
		\begin{equation}\label{empty intersection}
			\partial\Omega \cap \mathbb{B}^n_{1-\varepsilon} =\varnothing.
		\end{equation}
Suppose $\mathbb{B}^n_{1-\varepsilon} \not\subset \Omega$. Since $\mathbb{B}^n_{1-\varepsilon}$ is path connected, \eqref{empty intersection} implies $\mathbb{B}^n_{1-\varepsilon} \subset \overline{\Omega}^{\complement}$ (the complement of $\overline{\Omega}$). Let $\Omega_{\varepsilon}$ be the $\varepsilon$ neighborhood of $\Omega$. Then it follows that
		\begin{equation*}
			\Omega_{\varepsilon} \subset \mbox{the $\varepsilon$ neighborhood of $(\mathbb{B}^n_{1-\varepsilon})^{\complement}$} = (\mathbb{B}^n_{1-2\varepsilon})^{\complement}.
		\end{equation*}
On the other hand, by the condition $d_H(\Omega,\mathbb{B}^n)<\varepsilon$ we get $\mathbb{B}^n\subset \Omega_{\varepsilon}$. It follows that $\mathbb{B}^n \subset (\mathbb{B}^n_{1-2\varepsilon})^{\complement}$, which is a contradiction for any $\varepsilon\in (0,\frac{1}{2})$.
	\end{proof}

Recall the real ellipsoid $E(A)$ is defined in \eqref{real ellipsoid eq} for $A$ in the parameter space $\mathcal{A}$ as in \eqref{ellipsoid parameter space}.
We introduce an equivalence relation between two domains in $\mathbb{C}^n$:
\begin{equation*}
	\Omega_1\sim \Omega_2 \quad \mbox{ if $\Omega_1=\Phi(\Omega_2)$ for some complex affine transformation $\Phi$ in $\mathbb{C}^n$}.
\end{equation*}

We collect all real ellipsoids $E(A)$ and their affine images by setting
\begin{equation*}
	\mathcal{E}=\{\Omega\subset \mathbb{C}^n: \Omega=\Phi(E(A)) \mbox{ for some complex affine transformation $\Phi$ of $\mathbb{C}^n$ and some $A\in \mathcal{A}$} \}.
\end{equation*}
Note that by a theorem of Webster (see Corollary 5.7 in \cite{We}) for any $\Omega\in \mathcal{E}$, there exists a unique $A\in \mathcal{A}$ such that $\Omega\sim E(A)$.

\begin{prop}\label{stability of closedness to ball lemma}
	Suppose $\Omega\sim E(A)$ for some $A=(A_1,\cdots, A_n)\in \mathcal{A}$.
	For any $\varepsilon\in (0,\frac{1}{2})$, if $d_H(\Omega, \mathbb{B}^n)<\varepsilon$, then $|A_j|\leq \frac{\varepsilon}{1+\varepsilon^2}$ for $1\leq j\leq n$.
\end{prop}

\begin{proof}
	Denoting $x_j=\Real z_j$ and $y_j=\Imaginary z_j$, we can rewrite the defining function of $E(A)$ as
	
		$$f_A=1-|z|^2-\sum_{j=1}^n A_j(z_j^2+\oo{z_j}^2)
		=1-\sum_{j=1}^n (1+2A_j)x_j^2-\sum_{j=1}^n(1-2A_j)y_j^2.$$
	
	For any $p\in \mathbb{C}^n$, we set
	\begin{equation*}
		L_{p, x_n}=\{ \lambda (0,\cdots, 0, 1)^t+p\in E(A): \lambda \in \mathbb{R} \}.
	\end{equation*}
	Thus, $L_{p,x_n}$ (if nonempty) stands for the longest real line segment through $p$ along the $x_n$ direction and lying in $E(A)$. For any real line segment $L$, we denote its length by $\text{length}(L)$.
	
	Let $\Phi(z)= Mz+\xi$ be a complex affine transformation such that $\Omega=\Phi(E(A))$, where $M\in \text{GL}(n, \mathbb{C})$ and $\xi \in \mathbb{C}^n$. Then $\Phi(L_{p,x_n})$ is a real line segment lying in $\Omega$ and
	\begin{equation*}
		\length(\Phi(L_{p,x_n}))=|M_n|\cdot\length(L_{p,x_n}),
	\end{equation*}
	where $M_n$ is the last column vector in $M$. In particular, for any two points $p_1, p_2 \in \mathbb{C}^n$,
	\begin{equation*}
	 \length(\Phi(L_{p_1,x_n}))\leq \length(\Phi(L_{p_2,x_n})),	\mbox{whenever} \length(L_{p_1,x_n})\leq \length(L_{p_2,x_n}).
	\end{equation*}
By varying the point $p\in \mathbb{C}^n$, we get a family of parallel line segments (if nonempty), $\{L_{p,x_n}\}_{p\in \mathbb{C}^n}$, lying in $E(A)$. Among this family, $L_{0,x_n}$ achieves the largest length, as we can see from the expression of $f_A$.
	\begin{equation*}
		\length(L_{0,x_n})=\sup_{p\in \mathbb{C}^n} \length(L_{p,x_n})=\frac{2}{\sqrt{1+2A_{n}}}.
	\end{equation*}
	Therefore, $\Phi(L_{0,x_n})$ has the largest length among the family $\{\Phi(L_{p,x_n})\}_{p\in \mathbb{C}^n}$.
	\begin{equation*}
		\length(\Phi(L_{0,x_n}))=\sup_{p\in \mathbb{C}^n} \length(\Phi(L_{p,x_n}))=\frac{2|M_n|}{\sqrt{1+2A_{n}}}.
	\end{equation*}
On the other hand, we note the family $\{\Phi(L_{p,x_n})\}_{p\in \mathbb{C}^n}$ exhausts all line segments that are parallel to $\Phi(L_{0,x_n})$ and lying in $\Omega$. Consequently, $\Phi(L_{0,x_n})$ has the largest length among all line segments that are parallel to $\Phi(L_{0,x_n})$ and in $\Omega$.
	
Since $\Omega\sim E(A)$, we see $\Omega$ is convex. Furthermore, since $d_H(\Omega,\mathbb{B}^n)<\varepsilon$, by Lemma \ref{Hausdorff distance lemma} we get $\mathbb{B}^n_{1-\varepsilon}\subset \Omega \subset \mathbb{B}^n_{1+\varepsilon}$. Thus,
	\begin{equation}\label{length in xn direction}
		2(1-\varepsilon)\leq \length(\Phi(L_{0,x_n}))=\frac{2|M_n|}{\sqrt{1+2A_{n}}}
		\leq 2(1+\varepsilon).
	\end{equation}
	
	Similarly, by considering the real line segment lying in $E(A)$ along $y_n$ direction
	\begin{equation*}
		L_{p, y_n}=\{ \lambda (0,\cdots, 0, \sqrt{-1})^t +p\in E(A): \lambda \in \mathbb{R} \},
	\end{equation*}
	we can prove $\Phi(L_{0, y_n})$ has the largest length among all line segments that are parallel to $\Phi(L_{0, y_n})$ and lying in $\Omega$. Therefore,
	\begin{equation}\label{length in yn direction}
		2(1-\varepsilon)\leq \length(\Phi(L_{0,x_n}))=\frac{2|M_n|}{\sqrt{1-2A_{n}}}
		\leq 2(1+\varepsilon).
	\end{equation}
	
	Combing the inequalities \eqref{length in xn direction} and \eqref{length in yn direction}, we obtain
	\begin{equation*}
		\frac{1-\varepsilon}{1+\varepsilon}\leq \frac{\sqrt{1-2A_n}}{\sqrt{1+2A_n}} \leq \frac{1+\varepsilon}{1-\varepsilon}.
	\end{equation*}
	A straightforward computation shows $|A_n| \leq \frac{\varepsilon}{1+\varepsilon^2}$. Hence the proof is completed.
\end{proof}

We are now ready to prove Theorem \ref{local rigidity}. Note that the "if" implication follows directly from the transformation formula for Bergman kernels under biholomorphisms. It is sufficient to prove the "only if" implication.

By Theorem \ref{total degree thm}, $G=\Phi(E(A))$ for some complex affine transformation $\Phi$ of $\mathbb{C}^n$ and some real ellipsoid $E(A)$ as in \eqref{real ellipsoid eq} with $A\in \mathcal{A}$. Suppose $d_H(G,\mathbb{B}^n)<\delta$ for some $\delta \in (0,\frac{1}{2})$. Then Proposition \ref{stability of closedness to ball lemma} implies that $|A_j|\leq \frac{\delta}{1+\delta^2}$ for $1\leq j\leq n$. We also note that since $K$ is assumed to be algebraic, by the argument of Section 5 in \cite{EbenfeltXiaoXu2020algebraicity}, $K$ has no logarithmic singularity in its Fefferman asymptotic expansion (meaning the logarithmic term vanishes to infinite order at $\partial G$). In \cite{Hirachi93}, Hirachi has proved that the Bergman kernel of a real ellipsoids $E(A)$ has no logarithmic singularity in its Fefferman expansion if and only if $A=0$, provided that all $A_j$ for $1\leq j\leq n$ are sufficiently small. So if we choose $\delta$ to be small enough, then $E(A)= \mathbb{B}^n$ (i.e., $A=0$) and thus $G$ is $\mathbb{B}^n$ after the complex affine transformation $\Phi$ of $\mathbb{C}^n$.
\end{proof}

For our final result,  as before when the Bergman kernel $K$ is rational, we write $K=\frac{p}{q}$ for some polynomials $p$ and $q$ such that $\gcd(p,q)=1$. As in Remark \ref{rmk rational}, the total degree of $K$ is then given by $\max\{\deg q+1, \deg p \}$. Thus, part (a) in Theorem \ref{total degree thm} implies that $\max\{\deg q, \deg p \}\geq 2n+2$. Moreover, if $\max\{\deg q, \deg p \}= 2n+2$, then by part (a) in Theorem \ref{total degree thm} again, we conclude that the total degree of $K$ is $2n+3$. 
We therefore have the following corollary of Theorem \ref{total degree thm} and Theorem \ref{local rigidity}.

\begin{cor}
Let $G\subset \mathbb{C}^n$  $(n\geq 2)$ be a smoothly bounded  pseudoconvex domain and $K$  the Bergman kernel of $G$. If $K$ is rational, then the following hold:
	\begin{itemize}
		\item [(a)] The rational degree of $K$ is at least $2n+2$.
		\item [(b)] If the rational degree of $K$ equals $2n+2$, then $G$ is a real ellipsoid up to a complex affine transformation of $\mathbb{C}^n$.
	\end{itemize}
Moreover, if we in addition assume the Hausdorff distance $d_H(G, \mathbb{B}^n)$ is sufficiently small, then we have
	\begin{itemize}
		\item[(c)] 	The rational degree of $K$ equals $2n+2$ if and only if $G$ is the unit ball up to a complex affine transformation of $\mathbb{C}^n$.
	\end{itemize}
\end{cor}

\bibliographystyle{plain}
\bibliography{references}

\end{document}